\documentclass[12pt]{amsart}

\newtheorem{theorem}{Theorem}[section]
\newtheorem{proposition}[theorem]{Proposition}
\newtheorem{lemma}[theorem]{Lemma}
\newtheorem{corollary}[theorem]{Corollary}

\theoremstyle{definition}

\newcommand{\ZZ}{ \ensuremath{\mathbb{Z}}}

\newcommand{\conv}{\mathrm{conv}}

\def\cocoa{{\hbox{\rm C\kern-.13em o\kern-.07em C\kern-.13em o\kern-.15em A}}}

\begin{document}

\title[The numbers of edges of 5-polytopes]{The numbers of edges of 5-polytopes\\ with a given number of vertices}

\author{Takuya Kusunoki}
\address{
Takuya Kusunoki,
Department of Pure and Applied Mathematics,
Graduate School of Information Science and Technology,
Osaka University,
Suita, Osaka, 565-0871, Japan}
\email{t-kusunoki@ist.osaka-u.ac.jp}

\author{Satoshi Murai}
\address{
Satoshi Murai,
Department of Mathematics,
Faculty of Education
Waseda University
1-6-1 Nishi-Waseda, Shinjuku, Tokyo 169-8050, Japan}
\email{s-murai@waseda.jp}
\thanks{The second author was partially supported by KAKENHI16K05102.
}


\begin{abstract}
A basic combinatorial invariant of a convex polytope $P$ is its $f$-vector $f(P)=(f_0,f_1,\dots,f_{\dim P-1})$,
where $f_i$ is the number of $i$-dimensional faces of $P$.
Steinitz characterized all possible $f$-vectors of $3$-polytopes and Gr\"unbaum characterized the pairs given by the first two entries of the $f$-vectors of $4$-polytopes.
In this paper, we characterize the pairs given by the first two entries of the $f$-vectors of $5$-polytopes.
The same result was also proved by Pineda-Villavicencio, Ugon and Yost independently.
\end{abstract}

\maketitle

\section{Introduction}

The study of $f$-vectors of convex polytopes is one of the central research topic in convex geometry.
We call a $d$-dimensional convex polytope a \textit{$d$-polytope}.
For a convex polytope (or a polyhedral complex) $P$, we write $f_i(P)$ for the number of $i$-dimensional faces of $P$.
The \textit{$f$-vector} of a $d$-polytope $P$ is the vector $f(P)=(f_0(P),f_1(P),\dots,f_{d-1}(P))$.
In 1906, Steinitz characterized all possible $f$-vectors of $3$-polytopes (see \cite[\S 10.3]{Gr}).
While a characterization of $f$-vectors of $4$-polytopes is a big open problem in convex geometry,
for any $0 \leq i < j \leq 3$,
the following set was characterized by
Gr\"unbaum \cite{Gr},
Barnette \cite{Bar}
and Barnette--Reay \cite{BR}
(see also \cite[Theorem 3.9]{BL})
\[
\big\{(f_i(P),f_j(P)): P \mbox{ is a $4$-polytope}\big\}.
\]
Moreover,
Sj\"oberg and Ziegler \cite{SZ} recently characterize all possible values of the pairs $(f_0,f_{03})$ of flag face numbers of $4$-polytopes.
In this paper, we characterize all possible $(f_0,f_1)$ pairs of $5$-polytopes.

Let
$$\mathcal E^d=\{(f_0(P),f_1(P)): \mbox{ $P$ is a $d$-polytope}\}.$$
The set $\mathcal E^3$ was determined by Steinitz in 1906 who shows that
$$
\textstyle \mathcal E^3=\{ (v,e): \frac 3 2 v \leq e \leq 3v-6\}.$$
Note that, by Euler's relation, this actually determines all possible $f$-vectors of $3$-polytopes.
In higher dimensions, it is easy to see that any $d$-polytope $P$ satisfies
\begin{align}
\label{inequality}
\frac d 2 f_0(P) \leq f_1(P) \leq {f_0(P) \choose 2}.
\end{align}
Indeed, the first inequality follows since $f_1(P)$ equals to $\frac 1 2 $ times the sum of degrees of the vertices of $P$ and since each vertex of $P$ has degree $\geq d$.
Gr\"unbaum \cite[\S 10.4]{Gr} proved that the inequality \eqref{inequality} characterizes $\mathcal E^4$, with four exceptions.
More precisely, he proved the following statement.

\begin{theorem}[Gr\"unbaum]
\label{4poly}
\[
\mathcal E^4=\left\{ (v,e): 2 v \leq e \leq {v \choose 2} \right\}\setminus \big\{(6,12),(7,14),(8,17),(10,20)\big\}.
\]
\end{theorem}

In dimension $5$, the situation is more complicated.
The set $\mathcal E^5$ is close to the set of integer points satisfying \eqref{inequality},
but there are not only a finite list of exceptions but also an infinite family of exceptions.
Indeed, our main result is the following.

\begin{theorem}
\label{mainthm}
Let
$L=\{(v, \lfloor \frac 5 2 v +1\rfloor): v \geq 7\}$
and $G=\{ (8,20),(9,25),(13,35)\}$. 
Then
\[\mathcal E^5=
\left\{ (v,e): \frac  5 2 v \leq e \leq {v \choose 2}\right\} \setminus (L \cup G).
\]
\end{theorem}

Here $\lfloor a \rfloor$ denotes the integer part of a rational number $a$.
Note that it is not hard to see $(8,20) \not \in \mathcal E^5$ since a $5$-polytope $P$ with $f_1(P)=\frac 5 2 f_0(P)$ must be a simple polytope.
Also, $(9,25) \not \in \mathcal E^5$ was proved in \cite{PUY} recently.

The following table illustrates the shape of $\mathcal E^5$.

\begin{figure}[h]
{\unitlength 0.1in%
\begin{picture}(40.0000,45.6500)(16.0000,-67.0000)%
\put(20.0000,-59.0000){\makebox(0,0){$15$}}%
\put(20.0000,-49.0000){\makebox(0,0){$25$}}%
\put(20.0000,-39.0000){\makebox(0,0){$35$}}%
\put(20.0000,-29.0000){\makebox(0,0){$45$}}%
%
\special{pn 8}%
\special{pa 1800 6400}%
\special{pa 5600 6400}%
\special{fp}%
\special{sh 1}%
\special{pa 5600 6400}%
\special{pa 5533 6380}%
\special{pa 5547 6400}%
\special{pa 5533 6420}%
\special{pa 5600 6400}%
\special{fp}%
\special{pn 8}%
\special{pn 8}%
\special{pa 1600 7400}%
\special{pa 1600 7400}%
\special{ip}%
\special{pn 8}%
\special{pa 1800 7400}%
\special{pa 1805 7399}%
\special{pa 1808 7398}%
\special{ip}%
\special{pa 1843 7387}%
\special{pa 1851 7384}%
\special{ip}%
\special{pa 1885 7370}%
\special{pa 1890 7367}%
\special{pa 1892 7366}%
\special{ip}%
\special{pa 1925 7349}%
\special{pa 1932 7345}%
\special{ip}%
\special{pa 1964 7326}%
\special{pa 1965 7325}%
\special{pa 1970 7321}%
\special{pa 1970 7321}%
\special{ip}%
\special{pa 2001 7299}%
\special{pa 2007 7294}%
\special{ip}%
\special{pa 2036 7271}%
\special{pa 2042 7266}%
\special{ip}%
\special{pa 2070 7241}%
\special{pa 2070 7241}%
\special{pa 2075 7237}%
\special{pa 2076 7236}%
\special{ip}%
\special{pa 2102 7210}%
\special{pa 2108 7204}%
\special{ip}%
\special{pa 2134 7177}%
\special{pa 2135 7176}%
\special{pa 2139 7171}%
\special{ip}%
\special{pa 2164 7143}%
\special{pa 2169 7137}%
\special{ip}%
\special{pa 2193 7109}%
\special{pa 2198 7102}%
\special{ip}%
\special{pa 2221 7073}%
\special{pa 2225 7068}%
\special{pa 2226 7067}%
\special{ip}%
\special{pa 2248 7037}%
\special{pa 2253 7030}%
\special{ip}%
\special{pa 2274 7000}%
\special{pa 2279 6993}%
\special{ip}%
\special{pa 2300 6963}%
\special{pa 2300 6962}%
\special{pa 2304 6956}%
\special{ip}%
\special{pa 2325 6925}%
\special{pa 2329 6918}%
\special{ip}%
\special{pa 2349 6886}%
\special{pa 2353 6879}%
\special{ip}%
\special{pa 2372 6848}%
\special{pa 2375 6843}%
\special{pa 2376 6841}%
\special{ip}%
\special{pa 2395 6809}%
\special{pa 2399 6802}%
\special{ip}%
\special{pa 2417 6769}%
\special{pa 2421 6762}%
\special{ip}%
\special{pa 2439 6730}%
\special{pa 2443 6723}%
\special{ip}%
\special{pa 2460 6690}%
\special{pa 2464 6682}%
\special{ip}%
\special{pa 2481 6649}%
\special{pa 2485 6642}%
\special{ip}%
\special{pa 2502 6609}%
\special{pa 2505 6602}%
\special{ip}%
\special{pa 2522 6568}%
\special{pa 2525 6562}%
\special{pa 2525 6561}%
\special{ip}%
\special{pa 2541 6528}%
\special{pa 2545 6520}%
\special{ip}%
\special{pa 2561 6487}%
\special{pa 2564 6479}%
\special{ip}%
\special{pa 2579 6445}%
\special{pa 2583 6438}%
\special{ip}%
\special{pa 2598 6404}%
\special{pa 2601 6397}%
\special{ip}%
\special{pa 2616 6363}%
\special{pa 2619 6355}%
\special{ip}%
\special{pa 2634 6321}%
\special{pa 2635 6320}%
\special{pa 2638 6314}%
\special{ip}%
\special{pa 2652 6280}%
\special{pa 2655 6272}%
\special{ip}%
\special{pa 2669 6238}%
\special{pa 2672 6230}%
\special{ip}%
\special{pa 2687 6196}%
\special{pa 2689 6188}%
\special{ip}%
\special{pa 2703 6154}%
\special{pa 2705 6150}%
\special{pa 2706 6146}%
\special{ip}%
\special{pa 2720 6112}%
\special{pa 2723 6104}%
\special{ip}%
\special{pa 2736 6070}%
\special{pa 2739 6062}%
\special{ip}%
\special{pa 2753 6027}%
\special{pa 2755 6020}%
\special{ip}%
\special{pa 2769 5985}%
\special{pa 2770 5981}%
\special{pa 2771 5977}%
\special{ip}%
\special{pa 2784 5943}%
\special{pa 2785 5941}%
\special{pa 2787 5935}%
\special{ip}%
\special{pa 2800 5900}%
\special{pa 2800 5900}%
\special{ip}%
\special{pa 2800 5900}%
\special{pa 2845 5774}%
\special{pa 2850 5759}%
\special{pa 2855 5745}%
\special{pa 2860 5730}%
\special{pa 2865 5716}%
\special{pa 2870 5701}%
\special{pa 2875 5687}%
\special{pa 2925 5537}%
\special{pa 2930 5521}%
\special{pa 2935 5506}%
\special{pa 2940 5490}%
\special{pa 2945 5475}%
\special{pa 2950 5459}%
\special{pa 2955 5444}%
\special{pa 3005 5284}%
\special{pa 3010 5267}%
\special{pa 3015 5251}%
\special{pa 3020 5234}%
\special{pa 3025 5218}%
\special{pa 3030 5201}%
\special{pa 3035 5185}%
\special{pa 3085 5015}%
\special{pa 3090 4997}%
\special{pa 3095 4980}%
\special{pa 3100 4962}%
\special{pa 3105 4945}%
\special{pa 3110 4927}%
\special{pa 3115 4910}%
\special{pa 3165 4730}%
\special{pa 3170 4711}%
\special{pa 3175 4693}%
\special{pa 3180 4674}%
\special{pa 3185 4656}%
\special{pa 3190 4637}%
\special{pa 3195 4619}%
\special{pa 3245 4429}%
\special{pa 3250 4409}%
\special{pa 3255 4390}%
\special{pa 3260 4370}%
\special{pa 3265 4351}%
\special{pa 3270 4331}%
\special{pa 3275 4312}%
\special{pa 3325 4112}%
\special{pa 3330 4091}%
\special{pa 3340 4051}%
\special{pa 3350 4009}%
\special{pa 3355 3989}%
\special{pa 3405 3779}%
\special{pa 3410 3757}%
\special{pa 3420 3715}%
\special{pa 3430 3671}%
\special{pa 3435 3650}%
\special{pa 3485 3430}%
\special{pa 3490 3407}%
\special{pa 3495 3385}%
\special{pa 3500 3362}%
\special{pa 3505 3340}%
\special{pa 3510 3317}%
\special{pa 3515 3295}%
\special{pa 3535 3203}%
\special{pa 3540 3181}%
\special{pa 3545 3157}%
\special{pa 3565 3065}%
\special{pa 3570 3041}%
\special{pa 3575 3018}%
\special{pa 3580 2994}%
\special{pa 3585 2971}%
\special{pa 3590 2947}%
\special{pa 3595 2924}%
\special{pa 3645 2684}%
\special{pa 3650 2659}%
\special{pa 3655 2635}%
\special{pa 3660 2610}%
\special{pa 3665 2586}%
\special{pa 3670 2561}%
\special{pa 3675 2537}%
\special{pa 3700 2412}%
\special{pa 3702 2400}%
\special{fp}%
\special{pn 8}%
\special{pn 8}%
\special{pa 1600 7400}%
\special{pa 1605 7394}%
\special{ip}%
\special{pa 1628 7364}%
\special{pa 1630 7362}%
\special{pa 1633 7358}%
\special{ip}%
\special{pa 1657 7329}%
\special{pa 1662 7322}%
\special{ip}%
\special{pa 1686 7293}%
\special{pa 1690 7288}%
\special{pa 1691 7287}%
\special{ip}%
\special{pa 1714 7257}%
\special{pa 1719 7251}%
\special{ip}%
\special{pa 1743 7221}%
\special{pa 1745 7219}%
\special{pa 1748 7215}%
\special{ip}%
\special{pa 1772 7186}%
\special{pa 1775 7181}%
\special{pa 1776 7179}%
\special{ip}%
\special{pa 1800 7150}%
\special{pa 1805 7144}%
\special{ip}%
\special{pa 1828 7114}%
\special{pa 1830 7112}%
\special{pa 1833 7108}%
\special{ip}%
\special{pa 1857 7079}%
\special{pa 1862 7072}%
\special{ip}%
\special{pa 1886 7043}%
\special{pa 1890 7038}%
\special{pa 1891 7037}%
\special{ip}%
\special{pa 1914 7007}%
\special{pa 1919 7001}%
\special{ip}%
\special{pa 1943 6971}%
\special{pa 1945 6969}%
\special{pa 1948 6965}%
\special{ip}%
\special{pa 1972 6936}%
\special{pa 1975 6931}%
\special{pa 1976 6929}%
\special{ip}%
\special{pa 2000 6900}%
\special{pa 2005 6894}%
\special{ip}%
\special{pa 2028 6864}%
\special{pa 2030 6862}%
\special{pa 2033 6858}%
\special{ip}%
\special{pa 2057 6829}%
\special{pa 2062 6822}%
\special{ip}%
\special{pa 2086 6793}%
\special{pa 2090 6788}%
\special{pa 2091 6787}%
\special{ip}%
\special{pa 2114 6757}%
\special{pa 2119 6751}%
\special{ip}%
\special{pa 2143 6721}%
\special{pa 2145 6719}%
\special{pa 2148 6715}%
\special{ip}%
\special{pa 2172 6686}%
\special{pa 2175 6681}%
\special{pa 2176 6679}%
\special{ip}%
\special{pa 2200 6650}%
\special{pa 2205 6644}%
\special{ip}%
\special{pa 2228 6614}%
\special{pa 2230 6612}%
\special{pa 2233 6608}%
\special{ip}%
\special{pa 2257 6579}%
\special{pa 2262 6572}%
\special{ip}%
\special{pa 2286 6543}%
\special{pa 2290 6538}%
\special{pa 2291 6537}%
\special{ip}%
\special{pa 2314 6507}%
\special{pa 2319 6501}%
\special{ip}%
\special{pa 2343 6471}%
\special{pa 2345 6469}%
\special{pa 2348 6465}%
\special{ip}%
\special{pa 2372 6436}%
\special{pa 2375 6431}%
\special{pa 2376 6429}%
\special{ip}%
\special{pa 2400 6400}%
\special{pa 2405 6394}%
\special{ip}%
\special{pa 2428 6364}%
\special{pa 2430 6362}%
\special{pa 2433 6358}%
\special{ip}%
\special{pa 2457 6329}%
\special{pa 2462 6322}%
\special{ip}%
\special{pa 2486 6293}%
\special{pa 2490 6288}%
\special{pa 2491 6287}%
\special{ip}%
\special{pa 2514 6257}%
\special{pa 2519 6251}%
\special{ip}%
\special{pa 2543 6221}%
\special{pa 2545 6219}%
\special{pa 2548 6215}%
\special{ip}%
\special{pa 2572 6186}%
\special{pa 2575 6181}%
\special{pa 2576 6179}%
\special{ip}%
\special{pa 2600 6150}%
\special{pa 2605 6144}%
\special{ip}%
\special{pa 2628 6114}%
\special{pa 2630 6112}%
\special{pa 2633 6108}%
\special{ip}%
\special{pa 2657 6079}%
\special{pa 2662 6072}%
\special{ip}%
\special{pa 2686 6043}%
\special{pa 2690 6038}%
\special{pa 2691 6037}%
\special{ip}%
\special{pa 2714 6007}%
\special{pa 2719 6001}%
\special{ip}%
\special{pa 2743 5971}%
\special{pa 2745 5969}%
\special{pa 2748 5965}%
\special{ip}%
\special{pa 2772 5936}%
\special{pa 2775 5931}%
\special{pa 2776 5929}%
\special{ip}%
\special{pa 2800 5900}%
\special{pa 2800 5900}%
\special{ip}%
\special{pa 2800 5900}%
\special{pa 2810 5888}%
\special{pa 2815 5881}%
\special{pa 2825 5869}%
\special{pa 2830 5862}%
\special{pa 2850 5838}%
\special{pa 2855 5831}%
\special{pa 2865 5819}%
\special{pa 2870 5812}%
\special{pa 2890 5788}%
\special{pa 2895 5781}%
\special{pa 2905 5769}%
\special{pa 2910 5762}%
\special{pa 2930 5738}%
\special{pa 2935 5731}%
\special{pa 2945 5719}%
\special{pa 2950 5712}%
\special{pa 2970 5688}%
\special{pa 2975 5681}%
\special{pa 2985 5669}%
\special{pa 2990 5662}%
\special{pa 3010 5638}%
\special{pa 3015 5631}%
\special{pa 3025 5619}%
\special{pa 3030 5612}%
\special{pa 3050 5588}%
\special{pa 3055 5581}%
\special{pa 3065 5569}%
\special{pa 3070 5562}%
\special{pa 3090 5538}%
\special{pa 3095 5531}%
\special{pa 3105 5519}%
\special{pa 3110 5512}%
\special{pa 3130 5488}%
\special{pa 3135 5481}%
\special{pa 3145 5469}%
\special{pa 3150 5462}%
\special{pa 3170 5438}%
\special{pa 3175 5431}%
\special{pa 3185 5419}%
\special{pa 3190 5412}%
\special{pa 3210 5388}%
\special{pa 3215 5381}%
\special{pa 3225 5369}%
\special{pa 3230 5362}%
\special{pa 3250 5338}%
\special{pa 3255 5331}%
\special{pa 3265 5319}%
\special{pa 3270 5312}%
\special{pa 3290 5288}%
\special{pa 3295 5281}%
\special{pa 3305 5269}%
\special{pa 3310 5262}%
\special{pa 3330 5238}%
\special{pa 3335 5231}%
\special{pa 3345 5219}%
\special{pa 3350 5212}%
\special{pa 3370 5188}%
\special{pa 3375 5181}%
\special{pa 3385 5169}%
\special{pa 3390 5162}%
\special{pa 3410 5138}%
\special{pa 3415 5131}%
\special{pa 3425 5119}%
\special{pa 3430 5112}%
\special{pa 3450 5088}%
\special{pa 3455 5081}%
\special{pa 3465 5069}%
\special{pa 3470 5062}%
\special{pa 3490 5038}%
\special{pa 3495 5031}%
\special{pa 3505 5019}%
\special{pa 3510 5012}%
\special{pa 3530 4988}%
\special{pa 3535 4981}%
\special{pa 3545 4969}%
\special{pa 3550 4962}%
\special{pa 3570 4938}%
\special{pa 3575 4931}%
\special{pa 3585 4919}%
\special{pa 3590 4912}%
\special{pa 3610 4888}%
\special{pa 3615 4881}%
\special{pa 3625 4869}%
\special{pa 3630 4862}%
\special{pa 3650 4838}%
\special{pa 3655 4831}%
\special{pa 3665 4819}%
\special{pa 3670 4812}%
\special{pa 3690 4788}%
\special{pa 3695 4781}%
\special{pa 3705 4769}%
\special{pa 3710 4762}%
\special{pa 3730 4738}%
\special{pa 3735 4731}%
\special{pa 3745 4719}%
\special{pa 3750 4712}%
\special{pa 3770 4688}%
\special{pa 3775 4681}%
\special{pa 3785 4669}%
\special{pa 3790 4662}%
\special{pa 3810 4638}%
\special{pa 3815 4631}%
\special{pa 3825 4619}%
\special{pa 3830 4612}%
\special{pa 3850 4588}%
\special{pa 3855 4581}%
\special{pa 3865 4569}%
\special{pa 3870 4562}%
\special{pa 3890 4538}%
\special{pa 3895 4531}%
\special{pa 3905 4519}%
\special{pa 3910 4512}%
\special{pa 3930 4488}%
\special{pa 3935 4481}%
\special{pa 3945 4469}%
\special{pa 3950 4462}%
\special{pa 3970 4438}%
\special{pa 3975 4431}%
\special{pa 3985 4419}%
\special{pa 3990 4412}%
\special{pa 4010 4388}%
\special{pa 4015 4381}%
\special{pa 4025 4369}%
\special{pa 4030 4362}%
\special{pa 4050 4338}%
\special{pa 4055 4331}%
\special{pa 4065 4319}%
\special{pa 4070 4312}%
\special{pa 4090 4288}%
\special{pa 4095 4281}%
\special{pa 4105 4269}%
\special{pa 4110 4262}%
\special{pa 4130 4238}%
\special{pa 4135 4231}%
\special{pa 4145 4219}%
\special{pa 4150 4212}%
\special{pa 4170 4188}%
\special{pa 4175 4181}%
\special{pa 4185 4169}%
\special{pa 4190 4162}%
\special{pa 4210 4138}%
\special{pa 4215 4131}%
\special{pa 4225 4119}%
\special{pa 4230 4112}%
\special{pa 4250 4088}%
\special{pa 4255 4081}%
\special{pa 4265 4069}%
\special{pa 4270 4062}%
\special{pa 4290 4038}%
\special{pa 4295 4031}%
\special{pa 4305 4019}%
\special{pa 4310 4012}%
\special{pa 4330 3988}%
\special{pa 4335 3981}%
\special{pa 4345 3969}%
\special{pa 4350 3962}%
\special{pa 4370 3938}%
\special{pa 4375 3931}%
\special{pa 4385 3919}%
\special{pa 4390 3912}%
\special{pa 4410 3888}%
\special{pa 4415 3881}%
\special{pa 4425 3869}%
\special{pa 4430 3862}%
\special{pa 4450 3838}%
\special{pa 4455 3831}%
\special{pa 4465 3819}%
\special{pa 4470 3812}%
\special{pa 4490 3788}%
\special{pa 4495 3781}%
\special{pa 4505 3769}%
\special{pa 4510 3762}%
\special{pa 4530 3738}%
\special{pa 4535 3731}%
\special{pa 4545 3719}%
\special{pa 4550 3712}%
\special{pa 4570 3688}%
\special{pa 4575 3681}%
\special{pa 4585 3669}%
\special{pa 4590 3662}%
\special{pa 4610 3638}%
\special{pa 4615 3631}%
\special{pa 4625 3619}%
\special{pa 4630 3612}%
\special{pa 4650 3588}%
\special{pa 4655 3581}%
\special{pa 4665 3569}%
\special{pa 4670 3562}%
\special{pa 4690 3538}%
\special{pa 4695 3531}%
\special{pa 4705 3519}%
\special{pa 4710 3512}%
\special{pa 4730 3488}%
\special{pa 4735 3481}%
\special{pa 4745 3469}%
\special{pa 4750 3462}%
\special{pa 4770 3438}%
\special{pa 4775 3431}%
\special{pa 4785 3419}%
\special{pa 4790 3412}%
\special{pa 4810 3388}%
\special{pa 4815 3381}%
\special{pa 4825 3369}%
\special{pa 4830 3362}%
\special{pa 4850 3338}%
\special{pa 4855 3331}%
\special{pa 4865 3319}%
\special{pa 4870 3312}%
\special{pa 4890 3288}%
\special{pa 4895 3281}%
\special{pa 4905 3269}%
\special{pa 4910 3262}%
\special{pa 4930 3238}%
\special{pa 4935 3231}%
\special{pa 4945 3219}%
\special{pa 4950 3212}%
\special{pa 4970 3188}%
\special{pa 4975 3181}%
\special{pa 4985 3169}%
\special{pa 4990 3162}%
\special{pa 5010 3138}%
\special{pa 5015 3131}%
\special{pa 5025 3119}%
\special{pa 5030 3112}%
\special{pa 5050 3088}%
\special{pa 5055 3081}%
\special{pa 5065 3069}%
\special{pa 5070 3062}%
\special{pa 5090 3038}%
\special{pa 5095 3031}%
\special{pa 5105 3019}%
\special{pa 5110 3012}%
\special{pa 5130 2988}%
\special{pa 5135 2981}%
\special{pa 5145 2969}%
\special{pa 5150 2962}%
\special{pa 5170 2938}%
\special{pa 5175 2931}%
\special{pa 5185 2919}%
\special{pa 5190 2912}%
\special{pa 5210 2888}%
\special{pa 5215 2881}%
\special{pa 5225 2869}%
\special{pa 5230 2862}%
\special{pa 5250 2838}%
\special{pa 5255 2831}%
\special{pa 5265 2819}%
\special{pa 5270 2812}%
\special{pa 5290 2788}%
\special{pa 5295 2781}%
\special{pa 5305 2769}%
\special{pa 5310 2762}%
\special{pa 5330 2738}%
\special{pa 5335 2731}%
\special{pa 5345 2719}%
\special{pa 5350 2712}%
\special{pa 5370 2688}%
\special{pa 5375 2681}%
\special{pa 5385 2669}%
\special{pa 5390 2662}%
\special{pa 5410 2638}%
\special{pa 5415 2631}%
\special{pa 5425 2619}%
\special{pa 5430 2612}%
\special{pa 5450 2588}%
\special{pa 5455 2581}%
\special{pa 5465 2569}%
\special{pa 5470 2562}%
\special{pa 5490 2538}%
\special{pa 5495 2531}%
\special{pa 5505 2519}%
\special{pa 5510 2512}%
\special{pa 5530 2488}%
\special{pa 5535 2481}%
\special{pa 5545 2469}%
\special{pa 5550 2462}%
\special{pa 5570 2438}%
\special{pa 5575 2431}%
\special{pa 5585 2419}%
\special{pa 5590 2412}%
\special{pa 5600 2400}%
\special{fp}%
\put(28.0000,-66.0000){\makebox(0,0){$6$}}%
\put(36.0000,-66.0000){\makebox(0,0){$10$}}%
\put(46.0000,-66.0000){\makebox(0,0){$15$}}%
%
\special{pn 4}%
\special{sh 1}%
\special{ar 3000 5300 16 16 0 6.2831853}%
\special{sh 1}%
\special{ar 3000 5500 16 16 0 6.2831853}%
\special{sh 1}%
\special{ar 3000 5400 16 16 0 6.2831853}%
%
\special{pn 4}%
\special{sh 1}%
\special{ar 2800 5900 16 16 0 6.2831853}%
\put(20.0000,-54.0000){\makebox(0,0){$20$}}%
\put(20.0000,-44.0000){\makebox(0,0){$30$}}%
\put(20.0000,-34.0000){\makebox(0,0){$40$}}%
%
\special{pn 4}%
\special{sh 1}%
\special{ar 3200 5200 16 16 0 6.2831853}%
\special{sh 1}%
\special{ar 3200 5100 16 16 0 6.2831853}%
\special{sh 1}%
\special{ar 3200 5000 16 16 0 6.2831853}%
\special{sh 1}%
\special{ar 3200 4900 16 16 0 6.2831853}%
\special{sh 1}%
\special{ar 3200 4800 16 16 0 6.2831853}%
\special{sh 1}%
\special{ar 3200 4700 16 16 0 6.2831853}%
\special{sh 1}%
\special{ar 3200 4600 16 16 0 6.2831853}%
%
\special{pn 4}%
\special{sh 1}%
\special{ar 3400 5000 16 16 0 6.2831853}%
\special{sh 1}%
\special{ar 3400 4800 16 16 0 6.2831853}%
\special{sh 1}%
\special{ar 3400 4600 16 16 0 6.2831853}%
\special{sh 1}%
\special{ar 3400 4400 16 16 0 6.2831853}%
\special{sh 1}%
\special{ar 3400 4200 16 16 0 6.2831853}%
\special{sh 1}%
\special{ar 3400 4000 16 16 0 6.2831853}%
\special{sh 1}%
\special{ar 3400 3800 16 16 0 6.2831853}%
\special{sh 1}%
\special{ar 3400 3900 16 16 0 6.2831853}%
\special{sh 1}%
\special{ar 3400 4100 16 16 0 6.2831853}%
\special{sh 1}%
\special{ar 3400 4300 16 16 0 6.2831853}%
\special{sh 1}%
\special{ar 3400 4500 16 16 0 6.2831853}%
\special{sh 1}%
\special{ar 3400 4700 16 16 0 6.2831853}%
\special{sh 1}%
\special{ar 3400 4100 16 16 0 6.2831853}%
%
\special{pn 4}%
\special{sh 1}%
\special{ar 3600 4900 16 16 0 6.2831853}%
\special{sh 1}%
\special{ar 3600 4700 16 16 0 6.2831853}%
\special{sh 1}%
\special{ar 3600 4500 16 16 0 6.2831853}%
\special{sh 1}%
\special{ar 3600 4300 16 16 0 6.2831853}%
\special{sh 1}%
\special{ar 3600 4100 16 16 0 6.2831853}%
\special{sh 1}%
\special{ar 3600 3900 16 16 0 6.2831853}%
\special{sh 1}%
\special{ar 3600 3700 16 16 0 6.2831853}%
\special{sh 1}%
\special{ar 3600 3500 16 16 0 6.2831853}%
\special{sh 1}%
\special{ar 3600 3300 16 16 0 6.2831853}%
\special{sh 1}%
\special{ar 3600 3100 16 16 0 6.2831853}%
\special{sh 1}%
\special{ar 3600 2900 16 16 0 6.2831853}%
\special{sh 1}%
\special{ar 3600 4600 16 16 0 6.2831853}%
\special{sh 1}%
\special{ar 3600 4400 16 16 0 6.2831853}%
\special{sh 1}%
\special{ar 3600 4200 16 16 0 6.2831853}%
\special{sh 1}%
\special{ar 3600 4000 16 16 0 6.2831853}%
\special{sh 1}%
\special{ar 3600 3800 16 16 0 6.2831853}%
\special{sh 1}%
\special{ar 3600 3600 16 16 0 6.2831853}%
\special{sh 1}%
\special{ar 3600 3400 16 16 0 6.2831853}%
\special{sh 1}%
\special{ar 3600 3200 16 16 0 6.2831853}%
\special{sh 1}%
\special{ar 3600 3000 16 16 0 6.2831853}%
%
\special{pn 4}%
\special{sh 1}%
\special{ar 3800 4400 16 16 0 6.2831853}%
\special{sh 1}%
\special{ar 3800 4200 16 16 0 6.2831853}%
\special{sh 1}%
\special{ar 3800 4000 16 16 0 6.2831853}%
\special{sh 1}%
\special{ar 3800 3800 16 16 0 6.2831853}%
\special{sh 1}%
\special{ar 3800 3600 16 16 0 6.2831853}%
\special{sh 1}%
\special{ar 3800 3400 16 16 0 6.2831853}%
\special{sh 1}%
\special{ar 3800 3200 16 16 0 6.2831853}%
\special{sh 1}%
\special{ar 3800 3000 16 16 0 6.2831853}%
\special{sh 1}%
\special{ar 3800 2800 16 16 0 6.2831853}%
\special{sh 1}%
\special{ar 3800 2600 16 16 0 6.2831853}%
\special{sh 1}%
\special{ar 3800 2400 16 16 0 6.2831853}%
\special{sh 1}%
\special{ar 3800 2500 16 16 0 6.2831853}%
\special{sh 1}%
\special{ar 3800 2700 16 16 0 6.2831853}%
\special{sh 1}%
\special{ar 3800 2900 16 16 0 6.2831853}%
\special{sh 1}%
\special{ar 3800 3100 16 16 0 6.2831853}%
\special{sh 1}%
\special{ar 3800 3300 16 16 0 6.2831853}%
\special{sh 1}%
\special{ar 3800 3500 16 16 0 6.2831853}%
\special{sh 1}%
\special{ar 3800 3500 16 16 0 6.2831853}%
\special{sh 1}%
\special{ar 3800 3700 16 16 0 6.2831853}%
\special{sh 1}%
\special{ar 3800 3900 16 16 0 6.2831853}%
\special{sh 1}%
\special{ar 3800 4100 16 16 0 6.2831853}%
\special{sh 1}%
\special{ar 3800 4300 16 16 0 6.2831853}%
\special{sh 1}%
\special{ar 3800 4500 16 16 0 6.2831853}%
%
\special{pn 4}%
\special{sh 1}%
\special{ar 4000 4400 16 16 0 6.2831853}%
\special{sh 1}%
\special{ar 4000 4200 16 16 0 6.2831853}%
\special{sh 1}%
\special{ar 4000 4000 16 16 0 6.2831853}%
\special{sh 1}%
\special{ar 4000 3800 16 16 0 6.2831853}%
\special{sh 1}%
\special{ar 4000 3600 16 16 0 6.2831853}%
\special{sh 1}%
\special{ar 4000 3400 16 16 0 6.2831853}%
\special{sh 1}%
\special{ar 4000 3200 16 16 0 6.2831853}%
\special{sh 1}%
\special{ar 4000 3000 16 16 0 6.2831853}%
\special{sh 1}%
\special{ar 4000 2800 16 16 0 6.2831853}%
\special{sh 1}%
\special{ar 4000 2600 16 16 0 6.2831853}%
\special{sh 1}%
\special{ar 4000 2400 16 16 0 6.2831853}%
\special{sh 1}%
\special{ar 4000 4100 16 16 0 6.2831853}%
\special{sh 1}%
\special{ar 4000 3900 16 16 0 6.2831853}%
\special{sh 1}%
\special{ar 4000 3700 16 16 0 6.2831853}%
\special{sh 1}%
\special{ar 4000 3500 16 16 0 6.2831853}%
\special{sh 1}%
\special{ar 4000 3300 16 16 0 6.2831853}%
\special{sh 1}%
\special{ar 4000 3100 16 16 0 6.2831853}%
\special{sh 1}%
\special{ar 4000 2900 16 16 0 6.2831853}%
\special{sh 1}%
\special{ar 4000 2700 16 16 0 6.2831853}%
\special{sh 1}%
\special{ar 4000 2500 16 16 0 6.2831853}%
%
\special{pn 4}%
\special{sh 1}%
\special{ar 4200 4000 16 16 0 6.2831853}%
\special{sh 1}%
\special{ar 4200 3800 16 16 0 6.2831853}%
\special{sh 1}%
\special{ar 4200 3600 16 16 0 6.2831853}%
\special{sh 1}%
\special{ar 4200 3400 16 16 0 6.2831853}%
\special{sh 1}%
\special{ar 4200 3200 16 16 0 6.2831853}%
\special{sh 1}%
\special{ar 4200 3000 16 16 0 6.2831853}%
\special{sh 1}%
\special{ar 4200 2800 16 16 0 6.2831853}%
\special{sh 1}%
\special{ar 4200 2600 16 16 0 6.2831853}%
\special{sh 1}%
\special{ar 4200 2400 16 16 0 6.2831853}%
\special{sh 1}%
\special{ar 4200 3700 16 16 0 6.2831853}%
\special{sh 1}%
\special{ar 4200 3500 16 16 0 6.2831853}%
\special{sh 1}%
\special{ar 4200 3300 16 16 0 6.2831853}%
\special{sh 1}%
\special{ar 4200 3100 16 16 0 6.2831853}%
\special{sh 1}%
\special{ar 4200 2900 16 16 0 6.2831853}%
\special{sh 1}%
\special{ar 4200 2700 16 16 0 6.2831853}%
\special{sh 1}%
\special{ar 4200 2500 16 16 0 6.2831853}%
%
\special{pn 4}%
\special{sh 1}%
\special{ar 4400 3700 16 16 0 6.2831853}%
\special{sh 1}%
\special{ar 4400 3500 16 16 0 6.2831853}%
\special{sh 1}%
\special{ar 4400 3300 16 16 0 6.2831853}%
\special{sh 1}%
\special{ar 4400 3100 16 16 0 6.2831853}%
\special{sh 1}%
\special{ar 4400 2900 16 16 0 6.2831853}%
\special{sh 1}%
\special{ar 4400 2700 16 16 0 6.2831853}%
\special{sh 1}%
\special{ar 4400 2500 16 16 0 6.2831853}%
\special{sh 1}%
\special{ar 4400 3600 16 16 0 6.2831853}%
\special{sh 1}%
\special{ar 4400 3400 16 16 0 6.2831853}%
\special{sh 1}%
\special{ar 4400 3200 16 16 0 6.2831853}%
\special{sh 1}%
\special{ar 4400 3000 16 16 0 6.2831853}%
\special{sh 1}%
\special{ar 4400 2800 16 16 0 6.2831853}%
\special{sh 1}%
\special{ar 4400 2600 16 16 0 6.2831853}%
\special{sh 1}%
\special{ar 4400 2400 16 16 0 6.2831853}%
%
\special{pn 4}%
\special{sh 1}%
\special{ar 4400 3900 16 16 0 6.2831853}%
\special{sh 1}%
\special{ar 4600 3500 16 16 0 6.2831853}%
\special{sh 1}%
\special{ar 4600 3300 16 16 0 6.2831853}%
\special{sh 1}%
\special{ar 4600 3100 16 16 0 6.2831853}%
\special{sh 1}%
\special{ar 4600 2900 16 16 0 6.2831853}%
\special{sh 1}%
\special{ar 4600 2700 16 16 0 6.2831853}%
\special{sh 1}%
\special{ar 4600 2500 16 16 0 6.2831853}%
\special{sh 1}%
\special{ar 4600 3400 16 16 0 6.2831853}%
\special{sh 1}%
\special{ar 4600 3200 16 16 0 6.2831853}%
\special{sh 1}%
\special{ar 4600 3000 16 16 0 6.2831853}%
\special{sh 1}%
\special{ar 4600 2800 16 16 0 6.2831853}%
\special{sh 1}%
\special{ar 4600 2600 16 16 0 6.2831853}%
\special{sh 1}%
\special{ar 4600 2400 16 16 0 6.2831853}%
%
\special{pn 4}%
\special{sh 1}%
\special{ar 4800 3200 16 16 0 6.2831853}%
\special{sh 1}%
\special{ar 4800 3000 16 16 0 6.2831853}%
\special{sh 1}%
\special{ar 4800 2800 16 16 0 6.2831853}%
\special{sh 1}%
\special{ar 4800 2600 16 16 0 6.2831853}%
\special{sh 1}%
\special{ar 4800 2400 16 16 0 6.2831853}%
\special{sh 1}%
\special{ar 4800 3100 16 16 0 6.2831853}%
\special{sh 1}%
\special{ar 4800 2900 16 16 0 6.2831853}%
\special{sh 1}%
\special{ar 4800 2700 16 16 0 6.2831853}%
\special{sh 1}%
\special{ar 4800 2500 16 16 0 6.2831853}%
\special{sh 1}%
\special{ar 4800 2500 16 16 0 6.2831853}%
%
\special{pn 4}%
\special{sh 1}%
\special{ar 5000 2600 16 16 0 6.2831853}%
\special{sh 1}%
\special{ar 5200 2600 16 16 0 6.2831853}%
\special{sh 1}%
\special{ar 5000 2800 16 16 0 6.2831853}%
\special{sh 1}%
\special{ar 5000 2400 16 16 0 6.2831853}%
\special{sh 1}%
\special{ar 5200 2400 16 16 0 6.2831853}%
%
\special{pn 4}%
\special{sh 1}%
\special{ar 5000 3000 16 16 0 6.2831853}%
\special{sh 1}%
\special{ar 5000 2900 16 16 0 6.2831853}%
\special{sh 1}%
\special{ar 5000 2700 16 16 0 6.2831853}%
\special{sh 1}%
\special{ar 5000 2500 16 16 0 6.2831853}%
\special{sh 1}%
\special{ar 5200 2500 16 16 0 6.2831853}%
\special{sh 1}%
\special{ar 5200 2700 16 16 0 6.2831853}%
\special{sh 1}%
\special{ar 5200 2900 16 16 0 6.2831853}%
%
\special{pn 8}%
\special{ar 3000 5600 20 20 0.0000000 6.2831853}%
%
\special{pn 8}%
\special{ar 3200 5300 20 20 0.0000000 6.2831853}%
%
\special{pn 8}%
\special{ar 3400 5100 20 20 0.0000000 6.2831853}%
%
\special{pn 8}%
\special{ar 3600 4800 20 20 0.0000000 6.2831853}%
%
\special{pn 8}%
\special{ar 3800 4600 20 20 0.0000000 6.2831853}%
%
\special{pn 8}%
\special{ar 4000 4300 20 20 0.0000000 6.2831853}%
%
\special{pn 8}%
\special{ar 4200 4100 20 20 0.0000000 6.2831853}%
%
\special{pn 8}%
\special{ar 4400 3800 20 20 0.0000000 6.2831853}%
%
\special{pn 8}%
\special{ar 4600 3600 20 20 0.0000000 6.2831853}%
%
\special{pn 8}%
\special{ar 4800 3300 20 20 0.0000000 6.2831853}%
%
\special{pn 8}%
\special{ar 5000 3100 20 20 0.0000000 6.2831853}%
%
\special{pn 8}%
\special{ar 5200 2800 20 20 0.0000000 6.2831853}%
%
\special{pn 8}%
\special{ar 5400 2600 20 20 0.0000000 6.2831853}%
%
\special{pn 4}%
\special{sh 1}%
\special{ar 5400 2400 16 16 0 6.2831853}%
\special{sh 1}%
\special{ar 5600 2400 16 16 0 6.2831853}%
\special{sh 1}%
\special{ar 5400 2500 16 16 0 6.2831853}%
\put(20.0000,-24.0000){\makebox(0,0){$e$}}%
\put(56.0000,-65.0000){\makebox(0,0){$v$}}%
%
\special{pn 8}%
\special{pa 2200 6600}%
\special{pa 2200 2400}%
\special{fp}%
\special{sh 1}%
\special{pa 2200 2400}%
\special{pa 2180 2467}%
\special{pa 2200 2453}%
\special{pa 2220 2467}%
\special{pa 2200 2400}%
\special{fp}%
\put(57.1000,-22.5000){\makebox(0,0){$e=\frac 5 2 v$}}%
\put(37.0000,-22.5000){\makebox(0,0){$e=\frac 1 2 v(v-1)$}}%
%
\special{pn 8}%
\special{pa 2150 5900}%
\special{pa 2250 5900}%
\special{fp}%
\special{pa 2250 4900}%
\special{pa 2150 4900}%
\special{fp}%
\special{pa 2150 3900}%
\special{pa 2250 3900}%
\special{fp}%
\special{pa 2250 2900}%
\special{pa 2150 2900}%
\special{fp}%
%
\special{pn 8}%
\special{pa 2150 3400}%
\special{pa 2250 3400}%
\special{fp}%
\special{pa 2250 4400}%
\special{pa 2150 4400}%
\special{fp}%
\special{pa 2150 5400}%
\special{pa 2250 5400}%
\special{fp}%
%
\special{pn 8}%
\special{pa 3600 6450}%
\special{pa 3600 6350}%
\special{fp}%
\special{pa 2800 6350}%
\special{pa 2800 6450}%
\special{fp}%
\special{pa 4600 6450}%
\special{pa 4600 6350}%
\special{fp}%
%
\special{pn 4}%
\special{sh 1}%
\special{ar 4800 3400 16 16 0 6.2831853}%
\special{sh 1}%
\special{ar 4800 3400 16 16 0 6.2831853}%
%
\special{pn 8}%
\special{pa 3170 5420}%
\special{pa 3230 5420}%
\special{fp}%
\special{pa 3230 5420}%
\special{pa 3200 5370}%
\special{fp}%
\special{pa 3200 5370}%
\special{pa 3170 5420}%
\special{fp}%
%
\special{pn 8}%
\special{pa 4170 3920}%
\special{pa 4230 3920}%
\special{fp}%
\special{pa 4230 3920}%
\special{pa 4200 3870}%
\special{fp}%
\special{pa 4200 3870}%
\special{pa 4170 3920}%
\special{fp}%
%
\special{pn 8}%
\special{pa 3370 4920}%
\special{pa 3430 4920}%
\special{fp}%
\special{pa 3430 4920}%
\special{pa 3400 4870}%
\special{fp}%
\special{pa 3400 4870}%
\special{pa 3370 4920}%
\special{fp}%
%
\special{pn 8}%
\special{ar 4850 5220 20 20 0.0000000 6.2831853}%
%
\special{pn 8}%
\special{pa 4820 5440}%
\special{pa 4880 5440}%
\special{fp}%
\special{pa 4880 5440}%
\special{pa 4850 5390}%
\special{fp}%
\special{pa 4850 5390}%
\special{pa 4820 5440}%
\special{fp}%
\put(53.8000,-52.0000){\makebox(0,0){points in $L$}}%
\put(53.8000,-54.0000){\makebox(0,0){points in $G$}}%
\end{picture}}%
\center{Table 1: Table of $\mathcal E^5$}
\end{figure}

In the table, black dots represent points in $\mathcal E^5$, white circles and triangles represent points in $L$ and $G$ respectively.
For example, on the line $v=9$, $(9,23) \in L$ is presented by a white circle,
$(9,25)\in G$ is presented by a triangle,
and the possible numbers of edges are $24,26,27,\dots,36$. 

Theorem \ref{mainthm} was also independently proved by Pineda-Villavicencio, Ugon and Yost \cite{PUY2} by a different method.

It would be interesting to determine $\mathcal E^d$ for $d \geq 6$, and more generally to characterize the set $\{(f_i(P),f_j(P)):\mbox{$P$ is a $d$-polytope}\}$ for any  $0 \leq i<j <d$.
About the latter problem,
Sj\"oberg and Ziegler \cite{SZ} recently study the case when $i=0$ and $j=d-1$.

\section{sufficiency}
In this section, we prove the sufficiency part of Theorem \ref{mainthm}.
If a polytope $Q$ is the pyramid over a polytope $P$, then we have
\[f_0(Q)=f_0(P)+1 \mbox{ and }f_1(Q)=f_0(P)+f_1(P).\]
This simple fact and Theorem \ref{4poly} prove the next lemma.

\begin{lemma}
\label{2.1}
\[ \mathcal E^5 \supset \left\{(v,e): 3v-3 \leq e \leq {v \choose 2} \right\} \setminus \big\{ (7,18),(8,21),(9,25),(11,30)\big\}.\]
\end{lemma}

Let $P$ be a $d$-polytope.
The \textit{degree} $\deg v$ of a vertex $v$ of $P$ is the number of edges of $P$ that contain $v$.
We say that a vertex $v$ is \textit{simple} if $\deg v=d$.
Let $V(P)$ be the vertex set of $P$.

\begin{lemma}
\label{2.2}
If $P$ is a $5$-polytope such that $f_1(P) \leq 3f_0(P) -1$, then $P$ has a simple vertex.
\end{lemma}

\begin{proof}
Observe $\deg v \geq 5$ for any $v \in V(P)$.
Since $\frac 1 2 \sum_{ v \in V(P)} \deg v=f_1(P)=3f_0(P)-1$, there must exist a vertex $v \in V(P)$ of degree $<6$.
\end{proof}

Let
\[
\textstyle
X=\left\{(v,e): \frac 5 2 v \leq e \leq {v \choose 2}\right\} \setminus (L \cup G)\]
be the right-hand side of Theorem \ref{mainthm}, and let
\[X_k=\{(v,e) \in X: v=k\}\]
for $k \in \ZZ$.
We want to prove $\mathcal E^5 \supset X_k$ for all $k \geq 6$.
To prove this, we use truncations.
For a $5$-polytope $P$ and its vertex $v \in V(P)$,
we write $\mathrm{tr}(P,v)$ for a polytope obtained from $P$ by truncating the vertex $v$.
If $v$ is simple, then
\[ f_0(\mathrm{tr}(P,v))=f_0(P)+4 \mbox{ and } f_1(\mathrm{tr}(P,v))=f_1(P)+10.\]

\begin{lemma}
\label{2.3}
For $k \geq 6$ with $k \not \in \{8,9,13\}$,
if $\mathcal E^5 \supset X_k$ then $\mathcal E^5 \supset X_{k+4}$.
\end{lemma}

\begin{proof}
Since $k \not \in \{8,9,13\}$,
\[
\textstyle
X_k=\left\{(k,e): \frac 5 2 k \leq e \leq {k \choose 2}\right\} \setminus \left\{ \left(k, \left\lfloor \frac 5 2k +1\right\rfloor\right)\right\}.\]
By Lemma \ref{2.2}, for any $5$-polytope $P$ with $f_0(P)=k$ and $f_1(P) \leq 3k-1$, we can make a $5$-polytope $Q$ with
\[f_0(Q)=k+4 \mbox{ and }f_1(Q)=f_1(P)+10\]
by truncating a simple vertex from $P$.
Since $\mathcal E^5 \supset X_k$, this implies
\begin{align*}
\mathcal E^5 
&\supset
\left \{ (k+4,e+10): \frac 5 2 k \leq e \leq 3k-1\right\} \setminus \left\{ \left(k+4, \left\lfloor \frac 5 2 k +11 \right\rfloor\right)\right\}\\
&=
\left \{ \big((k+4),e'\big): \frac 5 2 (k+4) \leq e' \leq 3(k+4)-3 \right\} \setminus \left\{ \left(k+4, \left\lfloor \frac 5 2 (k+4) +1 \right\rfloor\right)\right\}.
\end{align*}
The above inclusion and Lemma \ref{2.1} prove the desired statement.
\end{proof}

Now, we prove the main result of this section.
For a convex polytope $P$,
we write $P^*$ for its dual polytope.
In the rest of the paper,
if a face of a convex polytope is a simplex, then we call it a \textit{simplex face}.
A face which is not a simplex is called a \textit{non-simplex face}.

\begin{theorem}
\label{sufficiency}
$\mathcal E^5 \supset X$.
\end{theorem}

\begin{proof}
By Lemma \ref{2.3},
it is enough to show that
\begin{align}
\label{2-1}
\mathcal E^5 \supset  X_6 \cup X_7 \cup X_8 \cup X_9 \cup X_{12} \cup X_{13} \cup X_{17}.
\end{align}
Let $\varphi(v)=3v-3$.
By Lemma \ref{2.1},
$(v,e) \in \mathcal E^5$ if $\varphi(v) \leq e \leq {v \choose 2}$ and $(v,e)\not \in \{(7,18),(8,21),(9,25),(11,30)\}$.
Observe $\varphi(6)=15,\varphi(7)=18,\varphi(8)=21,\varphi(9)=24,\varphi(12)=33,\varphi(13)=36,\varphi(17)=48$.
Then, to prove \eqref{2-1}, what we must prove is
\begin{align}
\label{2-2}
\mathcal E^5 \supset \{
(12,30),(12,32),(13,34),(17,44),(17,45),(17,46),(17,47)\}.
\end{align}
(See also Table 1.)
Note that this observation says $\mathcal E^5 \supset X_k$ for $k \leq 9$.

Let $C$ be the cyclic $5$-polytope with $7$ vertices.
Then $f(C)=(7,21,34,30,12)$ (see \cite[\S 18]{Br}).
Hence $f_0(C^*)=12$ and $f_1(C^*)=30$, and therefore $(12,30) \in \mathcal E^5$.
Let $C'$ be the polytope obtained from $C^*$ by truncating its vertex.
Note that every vertex of $C^*$ is simple.
Since a truncation of a simple vertex creates a simplex facet,
$C'$ contains a simplex facet $F$.
Let $C''$ be the polytope obtained from $C'$ by adding a pyramid over $F$. Then
\[f_0(C'')=f_0(C')+1=f_0(C^*)+5=17\]
and
\[f_1(C'')=f_1(C')+5=f_1(C^*)+15=45.\]
Hence $(17,45) \in \mathcal E^5$.
We already see $(8,22),(9,24) \in \mathcal E^5$.
Then, using truncations of simple vertices and Lemma \ref{2.2},
we have $(12,32),(13,34),(17,44) \in \mathcal E^5$.
Also, since $(13,36),(13,37) \in \mathcal E^5$,
by the same argument we have $(17,46),(17,47) \in \mathcal E^5$.
These complete the proof of the theorem.
\end{proof}

\section{Necessity}

In this section, we prove the necessity part of Theorem \ref{mainthm}.
We first show that any element in $L$ is not contained in $\mathcal E^5$.
We introduce some lemmas which we need.
The following fact appears in \cite[\S 6.1]{Gr} (see also \cite[Problem 6.8]{Zi}).

\begin{lemma}
\label{3.1}
There are exactly four combinatorially different $4$-polytopes with $6$ facets. They are
\begin{itemize}
\item[($P_A$)] Pyramid over a square pyramid;
\item[($P_B$)] Pyramid over a triangular prism;
\item[($P_C$)] A polytope obtained from a $4$-simplex by truncating its vertex;
\item[($P_D$)] Product of two triangles.
\end{itemize}
\end{lemma}

Here are Schlegel diagrams and a list of facets of $P_A,P_B,P_C$ and $P_D$.
\bigskip

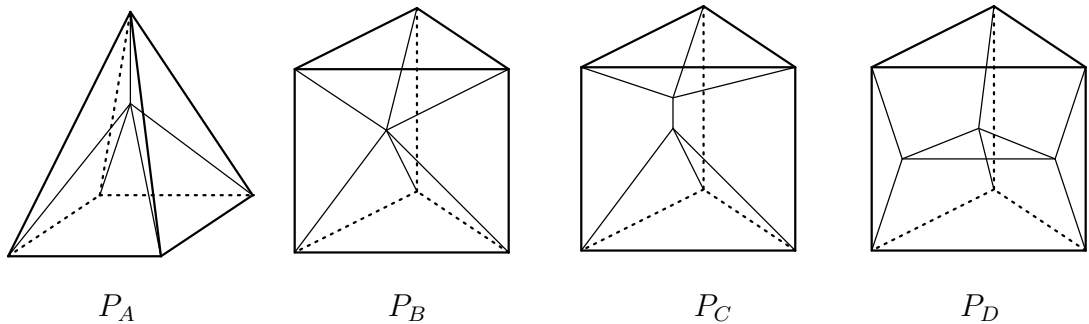
\begin{center}
{\unitlength 0.1in%
\begin{picture}(56.4000,15.1500)(16.0000,-24.8500)%
%
\special{pn 13}%
\special{pa 1600 2280}%
\special{pa 2400 2280}%
\special{fp}%
\special{pa 2400 2280}%
\special{pa 2880 1960}%
\special{fp}%
\special{pa 2880 1960}%
\special{pa 2880 1960}%
\special{fp}%
%
\special{pn 13}%
\special{pa 2880 1960}%
\special{pa 2880 1960}%
\special{dt 0.045}%
\special{pa 2080 1960}%
\special{pa 2080 1960}%
\special{dt 0.045}%
\special{pa 1600 2280}%
\special{pa 2080 1960}%
\special{dt 0.045}%
\special{pa 2080 1960}%
\special{pa 2880 1960}%
\special{dt 0.045}%
%
\special{pn 8}%
\special{pa 2080 1960}%
\special{pa 2080 1960}%
\special{fp}%
%
\special{pn 13}%
\special{pa 2240 1000}%
\special{pa 1600 2280}%
\special{fp}%
\special{pa 2400 2280}%
\special{pa 2240 1000}%
\special{fp}%
\special{pa 2240 1000}%
\special{pa 2880 1960}%
\special{fp}%
%
\special{pn 8}%
\special{pa 2240 1480}%
\special{pa 1600 2280}%
\special{fp}%
\special{pa 2400 2280}%
\special{pa 2240 1480}%
\special{fp}%
\special{pa 2240 1480}%
\special{pa 2880 1960}%
\special{fp}%
\special{pa 2080 1960}%
\special{pa 2240 1480}%
\special{fp}%
\special{pa 2240 1480}%
\special{pa 2240 1000}%
\special{fp}%
%
\special{pn 13}%
\special{pa 2080 1960}%
\special{pa 2240 1000}%
\special{dt 0.045}%
%
\special{pn 13}%
\special{pa 4600 2250}%
\special{pa 5720 2250}%
\special{fp}%
%
\special{pn 13}%
\special{pa 5720 2250}%
\special{pa 5240 1930}%
\special{dt 0.045}%
\special{pa 5240 1930}%
\special{pa 4600 2250}%
\special{dt 0.045}%
%
\special{pn 13}%
\special{pa 4600 1290}%
\special{pa 5720 1290}%
\special{fp}%
\special{pa 5720 1290}%
\special{pa 5240 970}%
\special{fp}%
\special{pa 5240 970}%
\special{pa 4600 1290}%
\special{fp}%
%
\special{pn 13}%
\special{pa 4600 1290}%
\special{pa 4600 2250}%
\special{fp}%
\special{pa 5720 2250}%
\special{pa 5720 1290}%
\special{fp}%
%
\special{pn 13}%
\special{pa 5240 970}%
\special{pa 5240 1930}%
\special{dt 0.045}%
%
\special{pn 8}%
\special{pa 5080 1450}%
\special{pa 5720 1290}%
\special{fp}%
\special{pa 5080 1450}%
\special{pa 5240 970}%
\special{fp}%
\special{pa 5080 1450}%
\special{pa 4600 1290}%
\special{fp}%
\special{pa 5080 1450}%
\special{pa 5080 1610}%
\special{fp}%
\special{pa 5080 1610}%
\special{pa 5240 1930}%
\special{fp}%
\special{pa 5080 1610}%
\special{pa 5720 2250}%
\special{fp}%
\special{pa 5080 1610}%
\special{pa 4600 2250}%
\special{fp}%
%
\special{pn 13}%
\special{pa 3100 2260}%
\special{pa 4220 2260}%
\special{fp}%
%
\special{pn 13}%
\special{pa 4220 2260}%
\special{pa 3740 1940}%
\special{dt 0.045}%
\special{pa 3740 1940}%
\special{pa 3100 2260}%
\special{dt 0.045}%
%
\special{pn 13}%
\special{pa 3100 1300}%
\special{pa 4220 1300}%
\special{fp}%
\special{pa 4220 1300}%
\special{pa 3740 980}%
\special{fp}%
\special{pa 3740 980}%
\special{pa 3100 1300}%
\special{fp}%
%
\special{pn 13}%
\special{pa 3100 1300}%
\special{pa 3100 2260}%
\special{fp}%
\special{pa 4220 2260}%
\special{pa 4220 1300}%
\special{fp}%
%
\special{pn 13}%
\special{pa 3740 980}%
\special{pa 3740 1940}%
\special{dt 0.045}%
%
\special{pn 13}%
\special{pa 6120 2250}%
\special{pa 7240 2250}%
\special{fp}%
%
\special{pn 13}%
\special{pa 7240 2250}%
\special{pa 6760 1930}%
\special{dt 0.045}%
\special{pa 6760 1930}%
\special{pa 6120 2250}%
\special{dt 0.045}%
%
\special{pn 13}%
\special{pa 6120 1290}%
\special{pa 7240 1290}%
\special{fp}%
\special{pa 7240 1290}%
\special{pa 6760 970}%
\special{fp}%
\special{pa 6760 970}%
\special{pa 6120 1290}%
\special{fp}%
%
\special{pn 13}%
\special{pa 6120 1290}%
\special{pa 6120 2250}%
\special{fp}%
\special{pa 7240 2250}%
\special{pa 7240 1290}%
\special{fp}%
%
\special{pn 13}%
\special{pa 6760 970}%
\special{pa 6760 1930}%
\special{dt 0.045}%
%
\special{pn 8}%
\special{pa 3580 1620}%
\special{pa 4220 1300}%
\special{fp}%
\special{pa 4220 2260}%
\special{pa 3580 1620}%
\special{fp}%
\special{pa 3580 1620}%
\special{pa 3740 1940}%
\special{fp}%
\special{pa 3580 1620}%
\special{pa 3100 2260}%
\special{fp}%
\special{pa 3580 1620}%
\special{pa 3100 1300}%
\special{fp}%
\special{pa 3580 1620}%
\special{pa 3740 980}%
\special{fp}%
%
\special{pn 8}%
\special{pa 6680 1610}%
\special{pa 6760 970}%
\special{fp}%
\special{pa 6760 1930}%
\special{pa 6680 1610}%
\special{fp}%
%
\special{pn 8}%
\special{pa 7080 1770}%
\special{pa 7240 1290}%
\special{fp}%
\special{pa 7240 2250}%
\special{pa 7080 1770}%
\special{fp}%
\special{pa 7080 1770}%
\special{pa 6280 1770}%
\special{fp}%
\special{pa 6280 1770}%
\special{pa 6680 1610}%
\special{fp}%
\special{pa 6680 1610}%
\special{pa 7080 1770}%
\special{fp}%
\special{pa 6280 1770}%
\special{pa 6120 1290}%
\special{fp}%
\special{pa 6280 1770}%
\special{pa 6120 2250}%
\special{fp}%
\put(21.7000,-25.5000){\makebox(0,0){$P_A$}}%
\put(36.9000,-25.5000){\makebox(0,0){$P_B$}}%
\put(52.9000,-25.5000){\makebox(0,0){$P_C$}}%
\put(66.9000,-25.5000){\makebox(0,0){$P_D$}}%
\end{picture}}%
\end{center}
\bigskip

\begin{center}
\begin{tabular}{|c|l|}
\hline
Type & Facets\\
\hline
$P_A$ & two square pyramids, four tetrahedra\\
\hline
$P_B$ & three square pyramids, one triangular prism, two tetrahedra\\
\hline
$P_C$ & four triangular prisms, two tetrahedra\\
\hline
$P_D$ & six triangular prisms\\
\hline
\end{tabular}
\end{center}
\bigskip

Recall that a convex polytope $P$ is said to be \textit{simplicial} if all its proper faces are simplices.
A \textit{simplicial $k$-sphere} is a simplicial complex whose geometric realization is homeomorphic to the $k$-sphere.
The boundary complex of a simplicial $d$-polytope is a simplicial $(d-1)$-sphere.
The next statement easily follows from the Lower Bound Theorem (see \cite{LBT}) and the Upper Bound Theorem (see \cite[Corollary II.3.5]{St}) for simplicial spheres.

\begin{lemma}
\label{3.2}
Let $\Delta$ be a simplicial $3$-sphere.
\begin{itemize}
\item[(i)] $f_3(\Delta) \ne 6,7,10$.
\item[(ii)] If $f_3(\Delta)=9$, then $\Delta$ is neighbourly, that is, every pair of  vertices of $\Delta$ are connected by an edge.
\end{itemize}
\end{lemma}

\begin{proof}
By the Lower Bound Theorem and the Upper Bound Theorem, we have
\begin{itemize}
\item if $f_0(\Delta)=5$, then $5 \leq f_3(\Delta) \leq 5$;
\item if $f_0(\Delta)=6$, then $8 \leq f_3(\Delta) \leq 9$;
\item if $f_0(\Delta)\geq 7$, then $11 \leq f_3(\Delta)$.
\end{itemize}
These clearly imply (i).
The statement (ii) follows from the fact that if the number of facets of a simplicial $(d-1)$-sphere equals to the bound in the Upper Bound Theorem, then it must be neighbourly (see e.g.\ the proof of \cite[Theorem 18.1]{Br}).
\end{proof}

We now prove that any element in $L$ is not contained in $\mathcal E^5$.

\begin{proposition}
\label{3.3}
If $P$ is a $5$-polytope, then
$f_1(P) \ne \lfloor \frac 5 2 f_0(P) +1 \rfloor$.
\end{proposition}

\begin{proof}
Suppose to the contrary that $f_1(P)=\lfloor \frac 5 2 f_0(P)+1\rfloor$.
We first consider the case when $f_0(P)$ is odd.
Then, since $\sum_{v \in V(P)} \deg v=2 f_1(P)=5 f_0(P)+1$,
$P$ has one vertex having degree $6$ and all other vertices have degree $5$.
Then $P^*$ has one facet $F$ with $f_3(F)=6$ and all other facets of $P^*$ are simplices.
However this implies that the $4$-polytope $F$ must be simplicial, which contradicts Lemma \ref{3.2}(i).

Next, we consider the case when $f_0(P)$ is even.
In this case,
$\sum_{v\in V(P)} \deg v= 2 f_1(P) = 5 f_0(P)+2$, so one of the following two cases occurs:
\begin{itemize}
\item[(a)] $P^*$ has one facet $F$ with $f_3(P)=7$ and all other facets of $P^*$ are simplices; 
\item[(b)] $P^*$ has two facets $F$ and $G$ with $f_3(F)=f_3(G)=6$ and all other facets of $P^*$ are simplices.
\end{itemize}
Since there are no simplicial $4$-polytope with $7$ facets by Lemma \ref{2.2}(i), the case (a) cannot occur.
Also, if the case (b) occurs, then $F$ and $G$ can have at most one non-simplex facet.
However, Lemma \ref{3.1} says that any $4$-polytope with $6$ facets have at least two non-simplex facets.
\end{proof}

Next, we show that any element of $G$ is not contained in $\mathcal E^5$.
Let
$\phi(v,d)= \frac 1 2 dv + \frac 1 2 (v-d-1)(2d-v)$.
The following result was proved by Pineda-Villavicencio, Ugon and Yost \cite[Theorems 6 and 19]{PUY}.

\begin{theorem}
Let $P$ be a $d$-polytope.
\begin{itemize}
\item[(i)] 
If $f_0(P) \leq 2d$, then $f_1(P) \geq \phi(f_0(P),d)$.
\item[(ii)] 
If $d \geq 4$, then $(f_0(P),f_1(P)) \ne (d+4, \phi(d+4,d)+1).$
\end{itemize}
\end{theorem}

By considering the special case when $d=5$ of the above theorem, we obtain the following.

\begin{corollary}
\label{3.4.1}
$(8,20),(9,25) \not \in \mathcal E^5$.
\end{corollary}

By Proposition \ref{3.3} and Corollary \ref{3.4.1},
to prove Theorem \ref{mainthm},
we only need to prove $(13,35) \not \in \mathcal E^5$.
We will prove this in the rest of this paper.

Let $P$ be a $5$-polytope.
For faces $F_1,F_2,\dots,F_k$ of $P$,
we write $\langle F_1,\dots,F_k\rangle$ for the polyhedral complex generated by $F_1,\dots,F_k$.
Let $\{G_1,\dots,G_l\}$ be a subset of the set of facets of $P$.
Then any $3$-face of $\Gamma=\langle G_1,\dots,G_l\rangle$ is contained in at most two facets of $\Gamma$.
We write
$$\partial \Gamma
=\langle H \in \Gamma: \mbox{ $H$ is a $3$-face of $\Gamma$ contained in exactly one facet of $\Gamma$}\rangle.$$
We often use the following trivial observation:
If $\{ G_1,\dots,G_l\}$ is the set of non-simplex facets of $P$, then $\partial \langle G_1,\dots,G_l\rangle$ is a simplicial complex.

We say that a $d$-polytope $P$ is \textit{almost simplicial} if all facets of $P$ except for one facet are simplices.
(We consider that simplicial polytopes are not almost simplicial.)
The next lemma can be checked by using a complete list of $4$-polytopes with at most $8$ vertices (see \cite{Mi}), but we write its proof for completeness.

\begin{lemma}
\label{3.5}
Let $P$ be a $4$-polytope.
\begin{itemize}
\item[(i)] Suppose that $P$ is almost simplicial and $f_3(P)=7$. Then $P$ is the pyramid over a triangular bipyramid.
\item[(ii)] Suppose that $P$ is almost simplicial and $f_3(P)=8$. Then $P$ does not contain a triangular bipyramid as a facet.
\item[(iii)] Suppose that $f_3(P)=7$ and $P$ has exactly two non-simplex facets $F$ and $G$.
Then none of $F$ and $G$ are square pyramids.
\end{itemize}
\end{lemma}

\begin{proof}
(i) Let $F$ be the unique non-simplex facet of $P$.
Clearly, $F$ is simplicial and $f_2(F)\leq f_3(P)-1=6$ since, for each $2$-face of $F$, there is a unique $3$-face of $P$ that contains it other than $F$.
Since a $3$-simplex and a triangular bipyramid are the only simplicial $3$-polytopes having at most $6$ facets, $F$ is a triangular bipyramid.
Then, since $P$ has $7$ facets,
$P$ must be the pyramid over $F$.

(ii)
Let $F$ be the unique non-simplex facet of $P$.
If $F$ is a triangular bipyramid, then by subdividing $F$ into two tetrahedra without introducing edges, one obtains a simplicial $3$-sphere $\Delta$ with $9$ facets.
Since $F$ is a triangular bipyramid, there are two vertices $u$ and $v$ of $F$ such that $u$ and $v$ are not connected by an edge in $F$.
These vertices are not connected by an edge in $\Delta$ by the construction of $\Delta$,
which contradicts Lemma \ref{3.2}(ii) saying that $\Delta$ must be neighbourly.

(iii)
Suppose to the contrary that $F$ is a square pyramid.
We claim that $G$ is also a square pyramid.
Indeed, since $\partial \langle F, G \rangle$ is a simplicial complex, $G$ contains exactly one non-simplex facet, and this facet must be a square and equals to $F \cap G$.
Also,
$$f_2(G) \leq f_3(P)-1=6.$$
Let $\Delta$ be a simplicial $2$-sphere obtained from $G$ by subdividing the square $F \cap G$ into two triangles.
Then $f_2(\Delta) \leq 7$, but since the number of $2$-faces of a simplicial $2$-sphere is even, $f_2(\Delta)=6$ and therefore $f_2(G)=5$.
This forces that $G$ is a square pyramid.

Let $F=\conv(v_1,a,b,c,d)$ and $G=\mathrm{conv}(v_2,a,b,c,d)$,
where $\conv(v_1,\dots,v_k)$ denotes the convex hull of points $v_1,\dots,v_k$.
Note that $\conv(a,b,c,d)$ is a square.
We assume that $\conv(a,c)$ and $\conv(b,d)$ are non-edges of $P$.
By subdividing each $F$ and $G$ into two tetrahedra by adding an edge $\conv(a,c)$, we can make a simplicial $3$-sphere $\Gamma$ with $f_3(\Gamma)=f_3(P)+2=9$.
By the construction of $\Gamma$,
$\conv(b,d)$ is not an edge of $\Gamma$, but this contradicts Lemma \ref{3.2}(ii).
\end{proof}

We also recall some known results on $h$-vectors of simplicial balls and their boundaries.
For a simplicial complex $\Delta$ of dimension $d-1$,
its \textit{$h$-vector} $h(\Delta)=(h_0(\Delta),h_1(\Delta),\dots,h_d(\Delta))$ is defined by
$$h_i(\Delta)=\sum_{j=0}^i (-1)^{i-j} {d-j \choose i-j} f_{j-1}(\Delta)$$
where $f_{-1}(\Delta)=1$.
A \textit{simplicial $d$-ball} is a simplicial complex whose geometric carrier is homeomorphic to a $d$-dimensional ball.
The following facts are known.
See \cite[Chapter II and Problem 12]{St}.

\begin{lemma}
\label{3.7}
Let $\Delta$ be a simplicial $d$-ball and $(h_0,h_1,\dots,h_{d+1})$ its $h$-vector. Then
\begin{itemize}
\item[(i)] 
$h_0=1$ and $h_{d+1}=0$.
\item[(ii)] $h_0+\cdots+h_{d+1}=f_d(\Delta)$.
\item[(iii)] $h_i \geq 0$ for all $i$.
\item[(iv)] The $h$-vector of the boundary complex of $\Delta$ is
$$(h_0,h_0+h_1-h_d,h_0+h_1+h_2-h_d-h_{d-1},\cdots).$$
\end{itemize}
\end{lemma}

We now complete the proof of Theorem \ref{mainthm}.

\begin{theorem}
\label{3.8}
$(13,35) \not \in \mathcal E^5$.
\end{theorem}

\begin{proof}
Suppose to the contrary that there is a $5$-polytope $P$ such that $f_0(P)=13$ and $f_1(P)=35$.
Let $v_1,\dots,v_{13}$ be the vertices of $P$ with $\deg v_1 \geq \cdots \geq \deg v_{13}$
and let $D=(\deg v_1,\deg v_2,\dots,\deg v_{13})$.
Since
$$\sum_{k=1}^{13} \deg v_k = 2 f_1(P)=70$$
and $\deg v_k \geq 5$ for each $k$,
$D$ must be one of the following.
\begin{itemize}
\item[(1)] $D=(10,5,\dots,5)$;
\item[(2)] $D=(9,6,5,\dots,5)$;
\item[(3)] $D=(8,7,5,\dots,5)$;
\item[(4)] $D=(8,6,6,5,\dots,5)$;
\item[(5)] $D=(7,7,6,5,\dots,5)$;
\item[(6)] $D=(7,6,6,6,5,\dots,5)$;
\item[(7)] $D=(6,6,6,6,6,5,\dots,5)$.
\end{itemize}
Below we will show a contradiction for each case.

(1) Suppose $D=(10,5,\dots,5)$.
Then $P^*$ has a $4$-face $F$ with $f_3(F)=10$.
This $F$ must be a simplicial $4$-polytope since all other facets of $P^*$ are simplices, which contradicts Lemma \ref{3.2}(i) saying that there are no simplicial $4$-polytopes with $10$ facets.

(2)
Suppose $D=(9,6,5,\dots,5)$.
$P^*$ has only two non-simplex $4$-faces $F$ and $G$.
These $4$-faces can have at most one non-simplex $3$-face.
By the assumption on $D$,
$F$ or $G$ must have $6$ facets.
This contradicts Lemma \ref{3.1} saying that any $4$-polytope with $6$ facets has at least two non-simplex facets.

(3)
Suppose $D=(8,7,5,\dots,5)$.
Let $F$ and $G$ be the $4$-faces of $P^*$ with $f_3(F)=8$ and $f_3(G)=7$.
Then $G$ is not simplicial by Lemma \ref{3.2}(i) and therefore $F$ is also not simplicial.
Hence $F$ and $G$ are almost simplicial and $F \cap G$ is a $3$-polytope which is not a simplex.
Since $f_3(F)=8$ and $f_3(G)=7$,
this contradicts Lemma \ref{3.5}(i) and (ii).

(4)
Suppose $D=(8,6,6,5,\dots,5)$.
Let $F,G$ and $G'$ be $4$-faces of $P^*$ with $f_3(F)=8$ and $f_3(G)=f_3(G')=6$.
Since $\partial \langle F,G,G'\rangle$ is a simplicial complex,
the polytopes $F,G$ and $G'$ have at most two non-simplex $3$-faces.
By Lemma \ref{3.1}, $G$ and $G'$ must be the polytope $P_A$, and $F$ must have $2$ square pyramids as its $3$-faces.
Thus, the $4$-polytope $F$ has $6$ tetrahedra and $2$ square pyramids as its facets. This implies
\[f_0(F^*)=f_3(F)=8 \ \mbox{ and } \ \ f_1(F^*)=f_2(F)= \frac 1 2 (5 \times 2 + 4 \times 6)=17.\]
However, this contradicts Theorem \ref{4poly} saying that $(8,17) \not \in \mathcal E^4$.

(5)
Suppose $D=(7,7,6,5,\dots,5)$.
Let $F,F'$ and $G$ be $4$-faces of $P^*$ with $f_3(F)=f_3(F')=7$ and $F(G)=6$.
Since all other $4$-faces of $P^*$ are simplices,
$F,F'$ and $G$ have at most two non-simplex $3$-faces.
Then, by Lemma \ref{3.1}, $G$ must be the polytope $P_A$, and $F$ and $F'$ have a square pyramid as its facets.
By Lemma \ref{3.5}(i), $F$ and $F'$ are not almost simplicial.
Hence $F$ and $F'$ have exactly two non-simplex facets, but this contradicts Lemma \ref{3.5}(iii).

(6)
Suppose $D=(7,6,6,6,5,\dots,5)$.
Let $F,G,G'$ and $G''$ be $4$-facets of $P^*$ with
$f_3(F)=7$ and $f_3(G)=f_3(G')=f_3(G'')=6$.
Since all other $4$-faces of $P^*$ are simplices, each of $F,G,G'$ and $G''$ can have at most three non-simplex $3$-faces.
By Lemma \ref{3.1}, $G,G'$ and $G''$ must be the polytope $P_A$ and have exactly two non-simplex $3$-faces.
Since $\partial \langle F,G,G',G''\rangle$ is a simplicial complex,
one of the following situations must occur:
\begin{itemize}
\item[(a)] $F$ is a simplicial polytope;
\item[(b)] $F$ has exactly two square pyramids as its facets and all other facets are simplices.
\end{itemize}
However, (a) cannot occur by Lemma \ref{3.2}(i)
and (b) cannot occur by Lemma \ref{3.5}(iii).

(7)
Suppose $D=(6,6,6,6,6,5,\dots,5)$.
Let $F_1,\dots,F_5$ be the $4$-faces of $P^*$ with $f_3(F_1)= \cdots =f_3(F_5)=6$.
Observing that each of $F_1,\dots,F_5$ must be one of $P_A,P_B,P_C$ and $P_D$ in Lemma \ref{3.1}.
It is not hard to see that one of the following situations must occur;
\begin{itemize}
\item[(a)] All the $F_1,\dots,F_5$ are $P_A$;
\item[(b)] All the $F_1,\dots,F_5$ are $P_C$;
\item[(c)] $F_1$ and $F_2$ are $P_B$.
$F_3,F_4$ and $F_5$ are $P_A$.
\end{itemize}

We first show that (a) cannot occur.
We may assume that $F_i \cap F_{i+1}$ is a square pyramid for $i=1,2,\dots,5$, where $F_6=F_0$.
Since $P_A$ has only one square as its $2$-faces,
$K=F_1\cap F_2 \cap \cdots \cap F_5$ must be this square, and the polyhedral complex generated by the facets of $P^*$ that do not contain $K$ is a shellable simplicial $4$-ball by a line shelling \cite[\S 8.2]{Zi}.
Let $B$ be this ball.
Clearly,
$$B=\langle G: G \mbox{ is a simplex facet of $P^*$}\rangle.$$
Hence $B$ has $13-5=8$ facets.
Also, since each $F_i \cap F_{i+1}$ is a pyramid over the square $K$, the faces $F_1,\dots,F_5$ can be written as
\[F_1=\conv(v_1,v_2,u_1,u_2,u_3,u_4),\]
\[F_2=\conv(v_2,v_3,u_1,u_2,u_3,u_4),\]
\[ \vdots \]
\[F_5=\conv(v_5,v_1,u_1,u_2,u_3,u_4)\]
with $K=\conv(u_1,\dots,u_4)$.
Then it follows that $\Gamma=\partial \langle F_1,\dots,F_5 \rangle$ is the join of the $5$-cycle and the $4$-cycle, and its $h$-vector is
$$(1,5,8,5,1)$$
since its entries coincide with the coefficients of the polynomial $(1+3t+t^2)(1+2t+t^2)$.
Let $h(B)=(1,h_1,h_2,h_3,h_4,0)$. Lemmas \ref{3.7} and \ref{3.8} say
$$1+h_1-h_4=5,\ 1+h_1+h_2-h_3-h_4=8 \mbox{ and } 1+h_1+h_2+h_3+h_4=8.$$ 
Then it is easy to see $h(B)=(1,4,3,0,0,0)$.
Let $G_1,\dots,G_8$ be a shelling of $B$.
Let 
$$R_j=\{v \in V(G_j): \conv(V(G_j)\setminus \{v\}) \in \langle G_1,\dots,G_{j-1}\rangle \}$$
and $S_j=V(G_j)\setminus R_j$,
where $V(G_j)$ is the vertex set of $G_j$.
Then
$$h_i=h_i(B)=|\{j: |R_j|=i\}|$$
for all $i$ (see \cite[\S 8.3]{Zi}). Since $h_3=h_4=0$,
we have $|S_j|=5-|R_j| \geq 3$ for all $j$.
By the definition of a shelling, $\conv(S_8)$ is a missing face of $\partial B$,
that is,
$\conv(S_8)$ is not a face of $\partial B$ but any its proper face is a face of $\partial B$.
Thus $\partial B$ has a missing face of dimension $\geq 2$.
However, the join of two cycles of length $\geq 4$ does not have any missing face of dimension $\geq 2$,
and
$\partial B=\partial \langle F_1,\dots,F_5\rangle$ is the join of the $5$-cycle and the $4$-cycle,
a contradiction.

We next prove that (b) cannot occur.
Observe that $P_C$ has $4$ non-simplex facets.
Since $\partial\langle F_1,\dots,F_5\rangle$ is a simplicial complex,
each $F_i \cap F_j$ must be a triangular prism.
Then it is easy to see that we can write
\begin{align*}
F_1 &= \conv(x_1,x_2,x_3,x_4,y_1,y_2,y_3,y_4),\\
F_2 &= \conv(x_1,x_2,x_3,x_5,y_1,y_2,y_3,y_5),\\
F_3 &= \conv(x_1,x_2,x_4,x_5,y_1,y_2,y_4,y_5),\\
F_4 &= \conv(x_1,x_3,x_4,x_5,y_1,y_3,y_4,y_5),\\
F_5 &= \conv(x_2,x_3,x_4,x_5,y_2,y_3,y_4,y_5),
\end{align*}
where each $\conv(x_i,x_j,x_k,y_i,y_j,y_k)$ is a triangular prism with triangles $\conv(x_i,x_j,x_k)$ and $\conv(y_i,y_j,y_k)$
(we assume that each $\conv(x_k,y_k)$ is an edge of $P^*$).
Using this formula,
one conclude that
\begin{align*}
&\partial \langle F_1,F_2,\dots,F_5 \rangle\\
&=\langle \conv(S): S \subset \{x_1,\dots,x_5\},\ |S|=4\rangle
 \cup
  \langle \conv(S): S \subset \{y_1,\dots,y_5\},\ |S|=4\rangle
\end{align*}
is the disjoint union of two copies of the boundary of a $4$-simplex.

If $\conv(x_1,\dots,x_5)$ is not a face of $P^*$, then, for each $S \subset \{x_1,\dots,x_5\}$ with $|S|=4$,
there is a unique $4$-face $G_S \not \in \{F_1,\dots,F_5\}$ of $P^*$ that contains $\conv(S)$.
This implies that, since $P^*$ has only 8 facets other than $F_1,\dots,F_5$, either $\conv(x_1,\dots,x_5)$ or $\conv(y_1,\dots,y_5)$ must be a face of $P^*$ (otherwise $P^*$ has at least 10 facets other than $F_1,\dots,F_5$).
We assume that $\conv(x_1,\dots,x_5)$  is a face of $P^*$.
Let $G_1,\dots,G_8$ be the simplex $4$-faces of $P^*$
and assume $G_1=\conv(x_1,\dots,x_5)$.
Then $\Gamma=\langle G_2,\dots,G_8\rangle$ is a pseudomanifold with
\[\partial \Gamma=\langle \conv(S): S \subset \{y_1,\dots,y_5\},\ |S|=4\rangle.\]
Let $n_i$ be the number of interior vertices of $\Gamma$.
By the Lower Bound Theorem for pseudomanifolds with boundary \cite{Fo} (see also \cite[Theorem 1.2]{Tay}), $\Gamma$ must have at least $5+4n_i-4$ facets.
Since $\Gamma$ only has $7$ facets, we have $n_i \leq 1$. However, this implies that $\Gamma$ is either the $4$-simplex or the cone over the boundary of the $4$-simplex, contradicting the fact that $\Gamma$ has $7$ facets.

We finally prove that (C) cannot occur.
Since $\partial \langle F_1,\dots,F_5\rangle$ is a simplicial complex,
$F_1 \cap F_2$ must be a triangular prism and $F_i \cap F_j$ is a square pyramid for all $i \in \{1,2\}$ and $j \in \{3,4,5\}$.
It is not hard to see that $F_1,\dots,F_5$ can be written as
\begin{align*}
F_1 &= \conv(x,y,z,x',y',z',v_1),\\
F_2 &= \conv(x,y,z,x',y',z',v_2),\\
F_3 &= \conv(x,y,x',y',v_1,v_2),\\
F_4 &= \conv(x,z,x',z',v_1,v_2),\\
F_5 &= \conv(y,z,y',z',v_1,v_2),
\end{align*}
where $\conv(x,y,z,x',y',z')$ is a triangular pyramid
with triangles $\conv(x,y,z)$ and $\conv(x',y',z')$ (we assume that $\conv(x,x'),\conv(y,y')$ and $\conv(z,z')$ are edges).
A routine computation shows
\begin{align*}
\partial \langle F_1,F_2,\dots,F_5 \rangle
=& \langle \conv(S): S \subset \{v_1,v_2,x,y,z\},\ |S|=4\rangle\\
&\cup \langle \conv(S): S \subset \{v_1,v_2,x',y',z'\},\ |S|=4\rangle
\end{align*}
is the union of two copies of the boundary of a $4$-simplex intersections in the edge $\conv(v_1,v_2)$.
Then the exactly same argument as in the case (b) works,
namely, one case show that either $\conv(v_1,v_2,x,y,z)$ or $\conv(v_1,v_2,x',y',z')$ must be a face of $P^*$ and conclude a contradiction by the Lower Bound Theorem for pseudomanifolds with boundary.
\end{proof}

\end{document}